\documentclass[12pt,draft]{amsart}

\usepackage{a4}
\usepackage[pstarrows]{pict2e}
\usepackage[utf8]{inputenc}
\usepackage[T1]{fontenc}

\DeclareMathOperator{\moda}{mod}

\long\def\comment#1\endcomment{}

\newtheorem{theorem}{Theorem}[section]
\newtheorem{prop}[theorem]{Proposition}
\newtheorem{lemma}[theorem]{Lemma}
\newtheorem{cor}[theorem]{Corollary}
\newtheorem{conj}[theorem]{Conjecture}
\newtheorem*{conjstar}{Conjecture}
\newtheorem*{question}{Question}
\newtheorem*{propstar}{Proposition}
\newtheorem*{claim}{Claim}
\theoremstyle{definition}
\newtheorem{defn}[theorem]{Definition}
\newtheorem{remark}[theorem]{Remark}
\newtheorem{example}[theorem]{Example}

\def\Bbb{\mathbb} \def\cal{\mathcal}
\def\wt#1{\widetilde{#1}}
\def\wl#1{\overline{#1}}
\def\sier#1{{\cal O}_{#1}}
\def\C{{\Bbb C}} \def\N{{\Bbb N}}
\def\cT{{\cal T}}

\newcommand{\bp}{\begin{picture}}
\newcommand{\bpn}{\begin{picture}(0,0)}
\newcommand{\ep}{\end{picture}}

\def\vir{\makebox(0,0){\rule{8pt}{8pt}}}
\def\cir{\circle*{0.23}}

\def\ci#1{%
{\bpn\put(0,0){\cir}\put(0,0.2){\makebox(0,0)[b]{$\scriptstyle #1$}}\ep}}
\def\vi#1{%
{\bpn\put(0,0){\vir}\put(0,0.3){\makebox(0,0)[b]{$\scriptstyle #1$}}\ep}}
\def\vit#1{%
{\bpn\put(0,0){\vir}\put(0.15,0){\makebox(0,0)[l]{$\scriptstyle #1$}}\ep}}
\def\cirr#1{%
{\bpn\put(0,0){\cir}\put(0.2,0){\makebox(0,0)[l]{$\scriptstyle #1$}}\ep}}

\def\lijn{\line(1,0){1}}

\def\keten{{\bpn\put(0,0){\line(1,0){0.5}}
 \put(1,0){\makebox(0,0){$\dots$}}\put(2,0){\line(-1,0){0.5}}\ep}}
\def\ketvert{{\bpn\put(0,0){\line(0,-1){0.5}}
 \put(0,-.7){\makebox(0,0){$\vdots$}}\put(0,-1.6){\line(0,1){0.5}}\ep}}

\begin{document}

\title{\bf Simple surface singularities}
\author{Jan Stevens}
\address{Department of Mathematical Sciences, Chalmers University of 
Technology and University of Gothenburg. 
SE 412 96 Gothenburg, Sweden}
\email{stevens@chalmers.se}

\begin{abstract}
By the famous ADE classification rational double points are simple. Rational 
triple points are also simple. We conjecture that 
the simple normal surface singularities are exactly those rational singularities,
whose resolution graph can be obtained from the graph of a rational
double point or rational triple point by making (some) vertex weights
more negative. For rational singularities we show one direction in general, 
and  the other direction (simpleness) within the special classes of rational 
quadruple points and of sandwiched singularities. 
\end{abstract}

\maketitle

\section*{Introduction}
Simple hypersurface singularities were classified by Arnol'd
in the famous ADE list \cite{ar}. In the surface case these are exactly the
rational double points. In Giusti's list \cite{gi} of simple isolated
complete intersection singularities no surface singularities occur.
They do appear
in the classification of simple determinantal codimension two singularities
by Fr\"uhbis-Kr\"uger and Neumer \cite{FN};  the simple surface 
singularities are the rational triple points.

In this paper we address the question:

\begin{question}
What are the simple normal surface singularities?
\end{question}

Simple  means here that
there occur only finitely many isomorphism classes in the versal deformation.
The known cases suggest that simple singularities are rational. 
There exist certainly simple singularities of higher embedding dimension, 
as all two-dimensional quotient singularities are simple (they deform only in 
other quotients \cite{EV}).
But not all rational singularities are simple: already for rational quadruple 
points there can be a cross ratio involved in the exceptional divisor. 
To exclude this we want the singularity to be taut. Furthermore it is natural 
to expect that the simple ones are quasi-homogeneous.
The quasi-homogeneous taut singularities  make up the parts I, II and III in
Laufer's list \cite{la1}. Their graphs have at most
one vertex of valency three and no higher valencies. 
There is a simple characterisation in terms of rational double point and triple 
point graphs. Using it we can formulate the conjectural  answer to our 
question as follows.

\begin{conjstar}
Simple normal surface singularities are exactly those rational singularities,
whose resolution graphs can be obtained from the graphs of rational
double points and rational triple points by making (some) vertex weights
more negative.
\end{conjstar}

A rigid singularity, one having no nontrivial
deformations at all, is certainly simple.
It is an old unsolved
question whether rigid normal surface singularities (or rigid reduced
curve singularities) exist. If they do, they are rather special. 
Our conjecture says not only that they do not exist, but even more,
 that there are no singularities
with nontrivial infinitesimal deformations, which however  all are  obstructed.

Without the condition of  normality there are rigid singularities. 
In fact, the standard
example of a nonnormal isolated singularity, two planes in $4$-space
meeting transversally in one point, is rigid and therefore simple.

The problem in studying rational singularities of multiplicity
at least four is that their deformation space has (in general) many
components, and for only one, the Artin component, one has good methods
to study adjacencies: it suffices to look at deformations of the
resolution. In the case of almost reduced fundamental cycle there is even
a complete description of the adjacencies \cite{la3}. 
For special classes we know more, mainly by the work of De Jong and Van Straten. 
This includes rational quadruple points \cite{tdq} and sandwiched  
singularities \cite{tds}.

Our results on the conjecture are restricted to rational singularities.
We prove one direction of the conjecture for rational singularities in general, that 
those not covered by the conjecture are not simple. We show that they deform 
on the Artin component to a singularity with a cross ratio on the exceptional 
set. Proving simpleness is more difficult. We succeed in the cases where 
there are good methods to study deformations.

Our result is:
\begin{propstar}
For the  following classes of rational singularities the conjecture is true, 
that is, the simple ones are those obtained from rational double and triple points:
\begin{itemize}
\item quotient singularities,
\item rational quadruple points,
\item sandwiched singularities. 
\end{itemize}
\end{propstar}

As singularities of type Laufer's III.5 (see Table \ref{tableA}) are deformations
of those of type III.6, and type III.7 and III.8 deformations of type 
III.9, it remains to prove that singularities with graph of type
III.9, of type III.6 and those of type III.4, which are not sandwiched, are simple.

The contents of this paper is as follows. We first recall the concepts of 
simpleness and modality. Then we discuss and state the conjecture, in 
particular why rationality is to be expected. We use the results of
Laufer \cite{la3} and Wahl \cite{wa} to determine adjacencies on the Artin 
component, giving one direction of the conjecture for rational singularities.
Thereafter we show that the conjecture holds for rational quadruple points.
The last section proves the result for sandwiched singularities. 

\section{modality}
Modality was introduced by Arnol'd in connection with the classification of 
singularities of functions under right equivalence.
It has been generalised to arbitrary actions of algebraic groups by Vinberg
\cite{vi}. Wall \cite{wall} described two possible generalisations for use 
in other classification problems in  singularity theory.

The first one is developed in detail by Greuel and Nguyen
\cite{GD}. Let an algebraic group $G$ act on an algebraic variety
$X$. Consider a \textit{Rosenlicht stratification} of $X$, that is
a stratification with locally closed $G$-invariant subvarieties $X_i$,
for which a geometric quotient $X_i/G$ exists. The \textit{modality}
$\moda(X,G)$ of the action is the maximal dimension of the quotients $X_i/G$.
For an open subset  $U$ the modality $\moda(U,G)$ is the maximal
dimension of the images of  $U\cap X_i$ in $X_i/G$. Finally,
the modality $\moda((X,x),G)$ of a point $x\in X$ is the minimum
of  $\moda(U,G)$ over open neighbourhoods $U$ of $x$. The modality is 
independent of the Rosenlicht stratification used.
One obtains right and contact modality of a germ by passing to the space
of $k$-jets, for sufficiently large $k$.

It is not obvious how to use this approach for modality of general
singularities, which are not complete intersections.
Wall's second description starts from the versal deformation 
$X_S\to S$ of a given singularity $X_0$.
The \textit{modality} is the
largest number $m$ such that in any neighbourhood of $0\in S$ there exists an 
$m$-dimensional analytic subset such that at most finitely many points
represent analytically isomorphic singularities.

The modality considered here is sometimes referred to as outer modality.
For right equivalence Gabrielov \cite{gab} has shown that outer modality and 
inner modality coincide, where 
inner or proper modality is the dimension of the $\mu$-constant
stratum in the semi-universal unfolding. For contact equivalence of
hypersurface germs the 
inner modality is in general smaller than the outer modality.
For general deformations one can use as inner modality the
dimension of the  \textit{modular stratum}, as introduced by
Palamodov \cite{pal}: the maximal subgerm of the base of the versal
deformation, over which the deformation is universal.

A singularity is \textit{simple} if it is 0-modal.  By the above discussion
this means the following.

\begin{defn}
A singularity is \textit{simple} if it only can deform into finitely many 
different isomorphism classes of singularities.
\end{defn}

A singularity is not simple if it deforms into a singularity with positive
inner modality, that is, a singularity with moduli.

\section{The conjecture}
The simple 2-dimensional hypersurface singularities are  the
rational double points. 
No complete intersection surface singularity, which is not a hypersurface, is simple: 
such a singularity deforms always into the intersection of two quadrics,
so in a simple elliptic singularity of type $\wt D_5$.
In the classification of simple determinantal codimension two singularities
by Fr\"uhbis-Kr\"uger and Neumer \cite{FN} the  surface 
singularities are the rational triple points; in particular non-rational 
singularities of this type are not simple.

Minimally elliptic singularities deform into simple elliptic singularities of the
same multiplicity \cite{ka}, and are therefore never simple. A general
elliptic singularity has on the resolution a minimally elliptic cycle, such that
its complement consists of rational cycles. We expect that a deformation
exists to a singularity with an elliptic curve as minimally elliptic cycle and
the same complementary rational cycles.  This would imply that elliptic 
singularities are never simple.
From Laufer's classification of taut and pseudotaut singularities \cite{la1} it
follows that a  \lq general\rq\ non-rational singularity  has moduli.
To explain what  \lq general\rq\ means we first recall the definition of
(pseudo)taut.

\begin{defn}
A normal surface singularity is \textit{taut}, if every other singularity
with the same resolution graph, is isomorphic to it. A singularity is 
\textit{pseudotaut} if there are only finitely many isomorphism classes of 
singularities with the same resolution graph.
\end{defn}
The only non-rational (pseudo)taut singularities are the minimally elliptic 
singularities whose graph is a Kodaira graph, and those are not simple.

Actually, Laufer defines a singularity as pseudotaut if there are  countably 
many isomorphism classes with the same graph, but he proves that there 
are then only  finitely many, as result of a more precise
description. 
Given a pseudotaut graph,  all analytic types
occur in a certain versal deformation:
construct by the standard plumbing construction a manifold $M$ 
with exceptional divisor $E$, and consider deformations of the
pair $(M,E)$, which are locally trivial deformations of $E$. 
For the general fibre all such deformations are in fact trivial.
Conversely, any singularity whose resolution can be deformed
in this way into one with only trivial such deformations, is 
pseudotaut. 
For all other singularities the resolution has moduli. This does not imply that 
the singularity itself has moduli, as deformations of the resolution blow 
down to deformations of the singularity if and only if the geometric genus 
$p_g$ is constant. So   \lq general\rq\  has to mean that $p_g$ has 
the lowest possible value.

We conjecture that simple singularities are rational. As such singularities are always
smoothable, the conjecture says in particular that no rigid normal surface singularities 
exist. Rigid  singularities, that is singularities with no deformations at all, not even 
infinitesimally,  are trivially simple.  In fact, the (non)-existence of rigid
normal surface singularities is an old open problem.

But not all rational singularities are simple. Already for quadruple points we 
find singularities with a modulus in the resolution. An example is the $n$-star 
singularity of \cite{tdq}, which has a star-shaped graph with a central
$(-4)$-vertex and four arms of $(-2)$'s of equal length $n-1$. On the other 
hand, there exist simple singularities of arbitrary multiplicity, as quotient 
singularities are simple: they deform only into other quotient singularities \cite{EV}.

What quotient singularities and rational triple points have in common, is 
that they are taut.
For rational singularities deformations of the resolution with the same
graph blow down to equisingular deformations of the singularity.
Therefore pseudotautness is a necessary condition for simplicity. 
Without using the classification we show below (Proposition \ref{notsimple})
that pseudotaut but not taut rational singularities are not simple.
From the same Proposition it follows that simple rational surface singularities are 
quasi-homogeneous. 

\begin{conj}
A  normal surface singularity is simple if and only if it is taut and
quasi-homogeneous.
\end{conj}

\begin{remark}
Quasi-homogeneity is a feature of many lists of simple objects, but in each 
case it is the result of the classification. In fact, the list of simple map germs 
$(\C,0)\to (\C^2,0)$ \cite{BrGa} contains also germs which are not 
quasi-homogeneous.
\end{remark}

The quasi-homogeneous taut singularities make up the parts I, II and III in
Laufer's list of graphs \cite{la1}.
We give the list in Table \ref{tableA}. It is 
organised as to have no duplicates. The meaning of the symbols is the
following.
A dot \bp(0.23,0.23)\put(0.12,0.12)\cir\ep\ denotes a vertex of any 
weight $-b\leq -2$, a dot\quad
$\ci{-2}$\quad
has weight exactly $-2$, whereas a square \rule{8pt}{8pt} is  a vertex  
of any weight $-b\leq -3$, less than $(-2)$. A chain
\bp(2,0)\put(0,0.12){\keten}
\put(2,0.12){\cir} 
\ep\quad is a chain of vertices of any length $k\geq 0$. So the first entry
of the Table gives exactly the cyclic quotient singularities, with  
the cone over a rational normal curve being the case $k=0$. The second entry 
gives all other quasi-homogeneous taut singularities with reduced fundamental 
cycle.

\begin{table}\caption{Quasi-homogeneous taut surface singularities}\label{tableA}
\begin{list}{}{}
\item[I, II] 
\[
\bp(2,0)(0,-.5)
\put(0,0){\cir}  
\put(0,0){\keten}
\put(2,0){\cir} 
\ep
\]
\item[III.1] 
\[
\bp(6,2.5)(0,-2.7)
\put(0,0){\cir}  \put(0,0){\keten}
\put(2,0){\cir} \put(2,0){\lijn}
\put(3,0){\vir}  
\put(3,0){\lijn}\put(4,0){\cir}\put(4,0){\keten}\put(6,0){\cir}
\put(3,0){\line(0,-1){1}} 
\put(3,-1){\cir}\put(3,-1){\ketvert}
\put(3,-2.6){\cir}
\ep
\]
\item[III.2] 
\[
\bp(4,1)(0,-1)
\put(0,0){\cir}  \put(0,0){\lijn}
\put(1,0){\ci{-2}}  
\put(1,0){\lijn}\put(2,0){\cir}\put(2,0){\keten}\put(4,0){\cir}
\put(1,0){\line(0,-1){1}} 
\put(1,-1){\cir}
\ep
\]

\item[III.3] 
\[
\bp(8,1.2)(0,-1.2)
\put(0,0){\cir}  \put(0,0){\keten}
\put(2,0){\cir} \put(2,0){\lijn}
\put(3,0){\vir}  \put(3,0){\lijn}
\put(4,0){\ci{-2}}  
\put(4,0){\lijn}\put(5,0){\cir}
\put(5,0){\lijn}\put(6,0){\cir}
\put(6,0){\keten}\put(8,0){\cir}
\put(4,0){\line(0,-1){1}} 
\put(4,-1){\cir}
\ep
\]
\item[III.4] 
\[
\bp(6,1.2)(0,-1.2)
\put(0,0){\cir}  \put(0,0){\lijn}
\put(1,0){\ci {-2}} \put(1,0){\lijn}
\put(2,0){\ci  {-2}}  \put(2,0){\lijn}
\put(3,0){\ci {-2}}  \put(3,0){\lijn}
\put(4,0){\cir}  
\put(4,0){\keten}\put(6,0){\cir}
\put(2,0){\line(0,-1){1}} \put(2,-1){\vir}
\ep
\]

\item[III.5] 
\[
\bp(7,1.2)(0,-1.2)
\put(0,0){\cir}  \put(0,0){\lijn}
\put(1,0){\ci {-2}} \put(1,0){\lijn}
\put(2,0){\ci  {-2}}  \put(2,0){\lijn}
\put(3,0){\ci {-2}}  \put(3,0){\lijn}
\put(2,0){\line(0,-1){1}} \put(2,-1){\cirr  {-2}}
\put(4,0){\vir}  \put(4,0){\lijn}
\put(5,0){\cir}\put(5,0){\keten}\put(7,0){\cir}
\ep
\]

\item[III.6] 
\[
\bp(8,1.2)(0,-1.2)
\put(0,0){\cir}  \put(0,0){\lijn}
\put(1,0){\ci {-2}} \put(1,0){\lijn}
\put(2,0){\ci  {-2}}  \put(2,0){\lijn}
\put(3,0){\ci {-2}}  \put(3,0){\lijn}
\put(4,0){\ci {-2}}  \put(4,0){\lijn}
\put(2,0){\line(0,-1){1}} \put(2,-1){\cirr  {-2}}
\put(5,0){\vir}  \put(5,0){\lijn}
\put(6,0){\cir}\put(6,0){\keten}\put(8,0){\cir}
\ep
\]
\item[III.7] 
\[
\bp(4,1.2)(0,-1.2)
\put(0,0){\cir}  \put(0,0){\lijn}
\put(1,0){\ci {-2}} \put(1,0){\lijn}
\put(2,0){\ci  {-2}}  \put(2,0){\lijn}
\put(3,0){\ci {-2}}  \put(3,0){\lijn}
\put(4,0){\cir}  
\put(2,0){\line(0,-1){1}} \put(2,-1){\cirr  {-2}}
\ep
\]
\item[III.8] 
\[
\bp(5,1.2)(0,-1.2)
\put(0,0){\cir}  \put(0,0){\lijn}
\put(1,0){\ci {-2}} \put(1,0){\lijn}
\put(2,0){\ci  {-2}}  \put(2,0){\lijn}
\put(3,0){\ci {-2}}  \put(3,0){\lijn}
\put(4,0){\ci {-2}}  \put(4,0){\lijn}
\put(2,0){\line(0,-1){1}} \put(2,-1){\cirr  {-2}}
\put(5,0){\cir}  
\ep
\]
\item[III.9] 
\[
\bp(6,1.2)(0,-1.2)
\put(0,0){\cir}  \put(0,0){\lijn}
\put(1,0){\ci {-2}} \put(1,0){\lijn}
\put(2,0){\ci  {-2}}  \put(2,0){\lijn}
\put(3,0){\ci {-2}}  \put(3,0){\lijn}
\put(4,0){\ci {-2}}  \put(4,0){\lijn}
\put(2,0){\line(0,-1){1}} \put(2,-1){\cirr  {-2}}
\put(5,0){\ci  {-2}}  \put(5,0){\lijn}
\put(6,0){\cir}
\ep
\]
\end{list}
\end{table}

Using the characterisation of the graphs in the list, already observed by 
Laufer, we get the  equivalent form of our conjecture, which was  stated
in the Introduction.

\begin{conj}
Simple normal surface singularities are exactly those rational singularities,
whose resolution graphs can be obtained from the graphs of rational
double points and rational triple points by making (some, or none) vertex 
weights more negative.
\end{conj}

We will say shortly that the singularities are \textit{obtainable} from 
rational double and triple points.
As the singularities under consideration are taut, their analytic structure is 
determined by the graph. For general singularities we get a relation between 
analytic structures by the following construction.

\begin{defn}
A singularity $Y$ is \textit{obtainable} from a singularity $X$, if the 
exceptional set on its minimal  
resolution $\wt Y$ is the strict transform of the exceptional set $E$ of the
minimal resolution $\wt X\to X$ under a blow up $\wt X' \to \wt X$ 
in smooth points of $E$.
\end{defn}

In this construction we allow that points are infinitely near, that is, we 
blow up successively in smooth points of the strict transform of $E$. 
In the terminology of \cite{le} the singularity $Y$ is a \textit{sandwiched 
singularity} relative to $X$, but of a special form, in that $X$ and  
$Y$ have the same underlying graph.

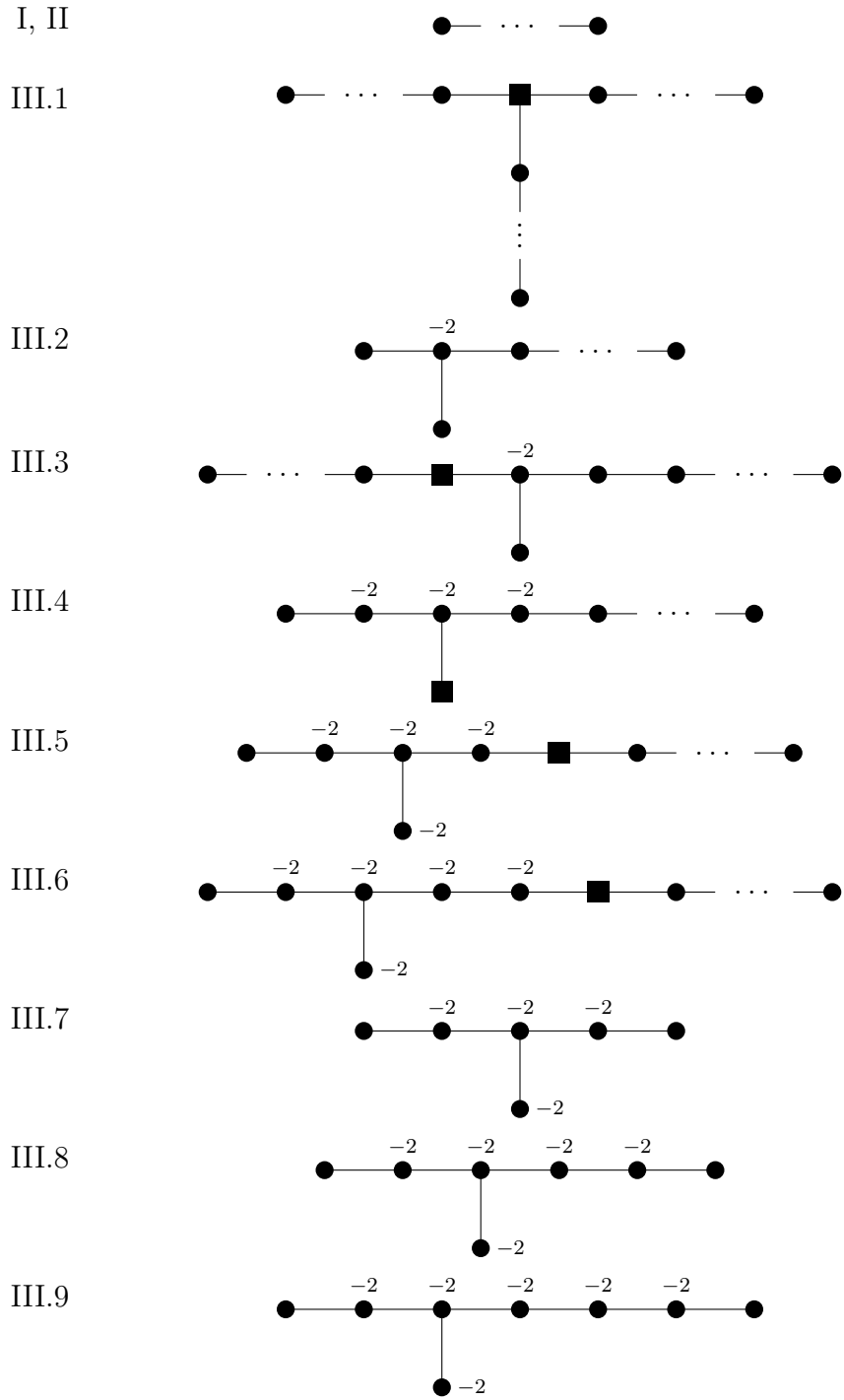
\comment
 A counterexample to the conjecture might be found
from a curve of high genus with a very special line bundle.
The Brill-Noether number $\rho =g-(r+1)(g-d+r)$ is the expected dimension
of the space $W_d^r(C)$ of complete linear systems of degree $d$ and 
dimension $r$ on a curve $C$ of genus $g$ (see e.g.\ \cite{ACGH}), and 
if this number is negative, only curves which are special in moduli 
have a $g_d^r$, but not much is then known about non-emptiness and 
dimension. In particular, if $\rho=3-3g$, then the number of curves together  
with a $g_d^r$ might be finite. Contracting the zero-section of the dual of  a
corresponding line bundle would then give a quasi-homogeneous surface 
singularity, which could be simple.
Note that by a 
result of Wahl \cite{wa} non-rational quasi-homogeneous singularities
always have a  positive dimensional equisingular stratum. For the putative 
example all equisingular deformations would have strictly positive weight.
\endcomment

\section{Adjacencies on the Artin component}
The deformation space of a rational surface singularity
has a special component, of largest dimension, over which
a simultaneous resolution exists (after base change).
Therefore the adjacencies on this component can be found from
deformations of the resolution. These were studied by 
Laufer and Wahl \cite{la3,wa}.

Let $X_T \to T$ be a one-parameter deformation with 
simultaneous resolution $\wt X_T\to T$. Consider an
irreducible component $E_T\to T$ of the exceptional set,
mapping surjectively onto $T$. Then $E_T$ is flat and proper
over $T$. The fibre $E_t$ for $t\neq0$ is irreducible
\cite[Thm 2.1]{la3}; in fact, if there are several 
components, there is monodromy and we find two components,
which are homologous in $\wt X_T$, contradicting the
negative definiteness of the intersection form on $X_t$.
In the family $E_t^2$ and $p_a(E_t)$ are constant; here
$p_a(D)$ is the arithmetic genus of a cycle $D$, which can be computed
with the adjunction formula. The 
divisor $E_t$ specialises for $t=0$ to a cycle $D$ on the
exceptional set of $X_0$. We say that $D$ \textit{lifts} to the
deformation $\wt X_T\to T$. Laufer and Wahl give sufficient
conditions on cycles $D$ for the existence of a deformation to
which $D$ lifts. To formulate the result we need some definitions.

\begin{defn}[\cite{wa}]
A cycle $D>0$  on the minimal resolution
of a rational singularity is a \textit{positive root}
if  $p_a(D)=0$.
\end{defn}

For the ADE singularities
the positive roots correspond  exactly to  the positive roots
of the root system of the same type  A, D or E under the identification of the 
resolution graph with the Dynkin diagram.

\begin{lemma}[{\cite[Lemma 3.1]{la3}}]
A cycle $D>0$ is a positive root if and only if $D$ is part of 
a computation sequence.
\end{lemma}
\begin{proof}
In a computation sequence the genus cannot go down:
let $Z_{j+1}=Z_j+E_{i(j)}$ be a step in the sequence, where
$Z_j\cdot E_{i(j)}>0$, then
\[
p_a(Z_{j+1})= p_a(Z_j)+p_a(E_{i(j)})+Z_j\cdot E_{i(j)}-1
\geq p_a(Z_j)\;.
\] 
So for a rational singularity $p_a(Z_j)=0$ in
every step.\\
Conversely, let $C\leq D$ be a cycle, which is part of a 
computation sequence and is maximal with respect to this property.
This implies that  $C\cdot E_i\leq 0$ for every $E_i$ in the
the support of $D-C$. For every cycle $A$ one has
$p_a(A)\leq 0$. If the support of $D-C$ is not empty,
then 
\[
0=p_a(D)=p_a(C)+p_a(D-C)+C\cdot (D-C)-1 <0\;.
\]
This contradiction shows that $C=D$.
\end{proof}

\begin{defn}
A cycle $D=\sum d_iE_i$ on the exceptional set of
the minimal resolution of a normal surface singularity is
\textit{almost reduced}, if it is reduced at the non-$(-2)$'s,
i.e., $d_i=1$ if  $E_i^2<-2$.
\end{defn}

This condition is important, as it implies  vanishing of the
obstructions to deform $D$, which lie in
$H^1(\cT^1_D)$, because
$h^1(\cT^1_D)=h^1(\sier{D-D_\text{red}}(D))
=(D-D_\text{red})\cdot K$
\cite[(2.14)]{wa}.

\begin{defn}[{\cite[Definition 3.5]{la3}}]
A collection of cycles $D_1$, \dots, $D_m$ on the exceptional set 
the minimal resolution of a rational surface singularity is 
\textit{integrally minimal}, if  each $D_i$ is a positive root,
the cycle $D=D_1+ \dots, +D_m$ is almost reduced, and 
no other collection $C_1$, \dots, $C_m$ of $m$ positive
roots can generate the $D_i$ with non-negative integral coefficients.
\end{defn}

\begin{theorem}\label{thmdefo}
Let $D_1$, \dots, $D_m$ be an integrally minimal collection
of cycles on the minimal resolution $\wt X_0$
of a rational surface singularity
with $D=\sum D_i$ itself a positive root. Then there exists
a 1-parameter deformation $ \wt X_T \to T$ such that the exceptional set 
$E_t$ of $\wt X_t$, $t\neq 0$, has a decomposition in $m$
irreducible components $E_{t,i}$ with $E_{t,i}$ homologous
to $D_i$ in $\wt X_T$.\\
If the fundamental cycle $Z$ of  $\wt X_0$ is itself
almost reduced, then every adjacency arises this way.
\end{theorem}

\begin{proof}
The result is contained in \cite[Thm 3.12]{la3}. Note that Laufer
has  almost reduced fundamental cycle as assumption, but his
proof shows that the existence of an integrally minimal collection
is sufficient for the  existence of a deformation.
\end{proof}

Some adjacencies, which Laufer calls \textit{reduced}, are particularly simple. 
In terms of the resolution graph they are the following.
First of all, one can take a (connected) subgraph of a given
graph. One can also replace two intersecting
curves $E_a$ and $E_b$ of self-intersection $-a$ and $-b$, with 
one curve with same self-intersection as $E_a+E_b$, that is
$-(a+b-2)$.  This means smoothing the double  point of the
exceptional divisor at the intersection point of  $E_a$ and $E_b$.

\begin{prop}\label{notsimple}
A rational singularity, which is not obtainable from a rational double or
triple point, is not simple.
It deforms into a singularity with a modulus in the exceptional divisor,
with\/ {\rm(}unweighted\/{\rm)} graph of the  form:
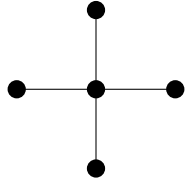
\[
\begin{picture}(2,2)(-1,-1)
\put(0,0){\line(1,0){1}}
\put(0,0){\line(-1,0){1}}
\put(0,0){\line(0,1){1}}
\put(0,0){\line(0,-1){1}}
\put(0,0){\cir}
\put(1,0){\cir}
\put(-1,0){\cir}
\put(0,1){\cir}
\put(0,-1){\cir}
\end{picture}
\]
\end{prop}
\begin{proof}
For the purpose of this proof we call a  graph as
in the statement a \textit{star}.
If the graph of a rational singularity has a vertex of valency at
least four, then it has a star as subgraph. If there are two
vertices $E_a$ and $E_b$ of valency three, then we can smooth all double points
of the exceptional divisor on the chain between $E_a$ and $E_b$.
In terms of the graph: we combine all vertices on the chain,
including $E_a$ and $E_b$ , into one new vertex. The new graph
has a star as subgraph. 

We are left with graphs with exactly one
vertex of valency three.
We claim that every such graph, which is not obtained from a double or 
triple point graph, has as subgraph a rational graph 
of form given in Table  \ref{tableB}. The meaning of the symbols is the 
same as in Table  \ref{tableA}. 
If all vertex weights are $(-2)$, then the graph is 
$\wt E_6$,  $\wt E_7$ or $\wt E_8$ (also known as the Kodaira graphs 
$\text{IV}^*$, $\text{III}^*$ or $\text{II}^*$). One has to make 
appropriate vertex weights more negative to obtain rational graphs.
We refer to the resulting graphs as graphs of type $\wt E_k$.
To prove the claim
one has only to carefully inspect Table  \ref{tableA}, and note which 
graphs are not there. First of all, the central curve has to be a $(-2)$-curve. 
If the three arms all have length at least three (counting from the central 
vertex), there is a subgraph of type $\wt E_6$. Otherwise, if there is one arm
of length two, and the other two have length at least four, then
the two vertices on the long arms next to the central vertex
have to be $(-2)$-vertices, and there is a subgraph of type $\wt E_7$.  
If not, the second arm has to have length 3, and the third
arm at least 6, and there are at least 6 $(-2)$-curves, so there 
is a subgraph of type $\wt E_8$. The Proposition follows from the following Lemma.
\end{proof}

\begin{table}\caption{Confining  non-simple  singularities}\label{tableB}
\begin{list}{}{}
\item[$\wt E_6$] 
\[
\bp(4,2)(0,-2.3)
\put(0,0){\cir}  \put(0,0){\lijn}
\put(1,0){\cir} \put(1,0){\lijn}
\put(2,0){\ci  {-2}}  \put(2,0){\lijn}
\put(3,0){\cir}  \put(3,0){\lijn}
\put(4,0){\cir}  
\put(2,0){\line(0,-1){1}} \put(2,-1){\cir}
\put(2,-1){\line(0,-1){1}} \put(2,-2){\cir}
\ep
\]
\item[$\wt E_7$] 
\[
\bp(6,1.2)(-1,-1.3)
\put(-1,0){\cir}  \put(-1,0){\lijn}
\put(0,0){\cir}  \put(0,0){\lijn}
\put(1,0){\ci {-2}} \put(1,0){\lijn}
\put(2,0){\ci  {-2}}  \put(2,0){\lijn}
\put(3,0){\ci {-2}}  \put(3,0){\lijn}
\put(4,0){\cir}  \put(4,0){\lijn}
\put(2,0){\line(0,-1){1}} \put(2,-1){\cir}
\put(5,0){\cir}  
\ep
\]
\item[$\wt E_8$] 
\[
\bp(7,1.2)(0,-1.3)
\put(0,0){\cir}  \put(0,0){\lijn}
\put(1,0){\ci {-2}} \put(1,0){\lijn}
\put(2,0){\ci  {-2}}  \put(2,0){\lijn}
\put(3,0){\ci {-2}}  \put(3,0){\lijn}
\put(4,0){\ci {-2}}  \put(4,0){\lijn}
\put(2,0){\line(0,-1){1}} \put(2,-1){\cirr  {-2}}
\put(5,0){\ci  {-2}}  \put(5,0){\lijn}
\put(6,0){\cir}\put(6,0){\lijn}
\put(7,0){\cir}
\ep
\]
\end{list}
\end{table}

\begin{lemma}\label{lemetilde}
Rational singularities with a graph  of type $\wt E_k$
deform into a star. 
\end{lemma}
\begin{proof}
We need some notation. We denote the central vertex
by  $E_0$. There are three arms, of length $p$, $q$ and $r$ (counted
from the central vertex). Here $(p,q,r)=(3,3,3)$, $(2,4,4)$ or $(2,3,6)$.
Referring to Table \ref{tableB} this means that $p$ is the length of the
arm pointing downwards, and $r$ the length of the right arm.
The  arms are $E_{1,1}+\dots+E_{1,p-1}$,  $E_{2,1}+\dots+E_{2,q-1}$ and
$E_{3,1}+\dots+E_{3,r-1}$, with  $E_{i,j}^2=-b_{i,j}$, and
$E_{i,1}$ intersecting  $E_0$. 

In each of the three cases we define a collection $\{D_0,\dots,D_4\}$
 of  integrally minimal cycles, such that the graph of the collection is a star
with $D_0$ as central vertex, that is,  $D_0\cdot D_i=1$ and $D_i\cdot D_j=0$ 
for $1\leq i<j\leq 4$.                                                 
Theorem \ref{thmdefo} then gives the existence of a deformation into a star 
with the given graph.
For graphs of type $\wt E_6$ we take
\begin{align*}
D_0&=E_0+E_{1,1}+E_{2,1}+E_{3,1}, & D_0^2&=-(b_{1,1}+b_{2,1}+b_{3,1}-4), \\
D_i&=E_{i,2},\quad1\leq i\leq 3, & D_i^2&=-b_{i,2},\\
D_4&=E_0, &D_4^2&=-2,
\end{align*}
for  $\wt E_7$ 
\begin{align*}
D_0&=E_{1,1}+E_{2,2}+E_{2,1}+E_0+E_{3,1}+E_{3,2}, & D_0^2&=-(b_{1,1}+b_{2,2}+b_{3,2}-4), \\
D_1&=E_0+E_{2,1},   & D_1^2&=-2,\\
D_2&= E_0+E_{3,1},  & D_2^2&=-2,\\
D_3&=E_{2,3},  & D_3^2&=-b_{2,3},\\
D_4&=E_{3,3}, &D_4^2&=-b_{3,3}
\end{align*}
and for $\wt E_8$ 
\begin{align*}
D_0&=E_{2,2}+E_{2,1}+E_0+E_{3,1}+E_{3,2}+E_{3,3}+E_{3,4}, & D_0^2&=-(b_{2,2}+b_{3,4}-2), \\
D_1&=E_{1,1}+E_0+E_{2,1},   & D_1^2&=-2,\\
D_2&= E_{1,1}+E_0+E_{3,1}+E_{3,2}+E_{3,3},  & D_2^2&=-2,\\
D_3&= E_{1,1}+E_{2,1}+2E_0+2E_{3,1}+E_{3,2},  & D_3^2&=-2,\\
D_4&=E_{3,5}&D_4^2&=-b_{3,5}.
\end{align*}

\end{proof}

\begin{remark}
In the special cases that $-b_{1,p-1}=-b_{2,q-1}=-b_{3,r-1}=-2$ for $\wt E_6$,
$-b_{2,q-1}=-b_{3,r-1}=-2$ for $\wt E_7$ and $-b_{3,r-1}=-2$ for $\wt E_8$,
the deformation has a quotient construction.
Then $-2K$ is an integral cycle, whose coefficients
are just the familiar multiplicities  of the $\wt E_k$-diagram.
The canonical cover  is a minimally elliptic singularity and
the  quotient of a
deformation to a simple elliptic singularity gives a deformation to 
a star. 
For $\wt E_8$ the multiplicities  are 
\[
\begin{matrix}2 &4&6&5&4&3&2&1\\
           &  & 3\end{matrix}
\]
and the double cover has a minimally elliptic graph  of type $\wt E_6$ with
$E_{1,2}^2=E_{2,2}^2=-b_{2,2}$ and $E_{3,2}^2=-(2b_{3,4}-2)$.
\end{remark}

\begin{cor}
Any non-simple rational singularity deforms into a star.
\end{cor}

\begin{proof}
If a singularity obtainable from a rational double or triple point only deforms 
into such singularities (necessarily only finitely many), then it
is simple. So if it is not simple, then it deforms to a rational singularity, not
of this type, and therefore also into a star.
\end{proof}

Next we study the adjacencies on the Artin component
of quasi-homogeneous  taut singularities. 
If the singularity is obtainable from a double point and the multiplicity is at least
at least four,  then it can also be  obtained from a triple point. Note that 
it can deform into double points.
We start by determining the positive roots. 

\begin{lemma}\label{rootslemma}
Let $X'$ be obtained from the rational triple point $X$, with the
irreducible components $E'_i$ of the exceptional divisor of $X'$
corresponding to the components $E_i$ of $X$.
A cycle $D'=\sum d_iE'_i$ is a positive root of $X'$ if and only if
$D=\sum d_iE_i$ is a positive root of $X$ with $d_i=1$ for all $i$
with ${(E'_i)}^2<E_i^2$.
\end {lemma}
\begin{proof}
We first remark that an $E_i'$ with ${(E'_i)}^2=E_i^2-\beta_i<E_i^2$ has
coefficient 1 in the fundamental cycle.
Indeed, by construction the exceptional divisor $E'$
is a subset of the exceptional set of a non-minimal resolution
of $X$; the fundamental cycle on this resolution can be computed by first 
computing the fundamental cycle on $E'$,
and by rationality each $(-1)$-curve intersects this cyle with
multiplicity 1. Therefore the condition on $d_i$ is necessary.

Let $D'$ be a positive root, so $p_a(D')=0$. We compute
$p_a(D)=1+\frac12D\cdot(D+K)$. If $d_i>1$, then
$E_i\cdot D= d_iE_i^2+\sum_{E_i\cdot E_j>0}d_j=
d_i{(E'_i)}^2+\sum_{E'_i\cdot E'_j>0}d_j=
E'_i\cdot D'$ and $E_i\cdot K = E'_i\cdot K'$.
If $d_i=1$ and ${(E'_i)}^2=E_i^2-\beta_i$, then
$E_i\cdot(D+K)=E'_i\cdot D'+\beta_i+E'_i\cdot K'-\beta_i=
E'_i\cdot (D'+K')$. So also $p_a(D)=0$.
The same computation shows that $p_a(D)$ implies $p_a(D')=0$, if
$d_i=1$ whenever $\beta_i>0$.
\end {proof}

\begin{prop}\label{propartin}
A singularity, obtainable from a rational triple point,
deforms on the Artin component only into singularities,
obtainable from triple and double points.
\end {prop}

\begin{proof}
Let $X'$ deform into a surface with several singularities. 
By openess of versality we can
smooth all but one of them, so we may as well assume that 
there is only one singularity.
As $X'$ has almost reduced
fundamental cycle, the deformation $X'_T$ can be described by an
integrally minimal collection $D_1'$, \dots, $D_m'$
of positive roots. 
Lemma \ref{rootslemma} gives an 
integrally minimal collection $D_1$, \dots, $D_m$, determining a deformation 
$X_T$ of the triple point $X$. As triple points deform only into triple or 
double points,  the graph of $D$ is a double or triple point graph. An $E_i'$ 
with ${(E'_i)}^2=E_i^2-\beta_i<E_i^2$
can occur in at most one $D_j'$, as its coefficient in $D'$ is one.
Therefore ${(D'_j)}^2=D_j^2-\sum \beta_i(j)$, where the sum runs over all 
$i$ such that $E_i'$ is contained in the support of $D_j$. So the graph of
$D'$ is obtainable from that of $D$.
 \end {proof}

\section{Rational quadruple points}
The deformation theory of rational quadruple points was studied by De Jong and 
Van Straten  \cite{tdq}. The base space of the versal deformation is up to a 
smooth factor isomorphic to a explicitly described space $B(n)$ with
$n+1$ irreducible components. The integer $n$ can 
be found from the resolution graph of the rational quadruple points, as it is the 
number of virtual quadruple points, or in other words the number of quadruple 
points in the resolution process.

If the base space has only two components, then a deformation to any other 
quadruple point has to occur over the intersection of the two components. So 
in particular it is a deformation on the Artin component.

\begin{prop}\label{proprqp}
A rational quadruple point, obtained from a triple point, is simple.
\end{prop}
\begin{proof}
By  Proposition \ref{propartin} any other quadruple point on the 
Artin component is obtainable from a triple point.
 Therefore simpleness follows if
the base space is $B(1)$, with exactly two components.

We show that $n=1$, that is, there is no quadruple point on the first blow up.
Let $E_0$ be the $(-3)$-curve of the triple point $X$, and $E_m$ the unique curve
of the quadruple point $X'$ 
with $E_m'\cdot E_m'< E_m\cdot E_m$  (possibly $E_m=E_0$).  
There  is a quadruple point on the first blow up of $X'$ if and only if $E_i'\cdot Z'=0$
for every curve  $E_i'$ on the chain from $E_0'$ to $E_m'$.

The multiplicity of $E_0'$ in the fundamental cycle $Z'$ is 1.
If  $E_0'\cdot Z'=0$ on $\wt X'$, then the neighbour $E_1'$
of $E_0'$ on the chain has the same multiplicity in $Z'$ as $E_1$ in $Z$,
and if $E_1'\cdot Z'=0$, its neighbour has also the same multiplicity, and so 
on.  This process stops with an  $E_i'\cdot Z'<0$, or reaches $E_m'$,
and $E_m'$ has the same multiplicity  in $Z'$ as 
$E_m$ in $Z$. But then $E_m'\cdot Z'<0$. 
\end{proof}

\begin{remark}
We can characterise the
simple rational quadruple points as those with almost reduced fundamental cycle 
and without quadruple points on the first blow up.
The easiest way to see this is to check the classification of quadruple points 
\cite[Prop. 4]{st-p}, as 
it is not always immediately obvious which of the two $(-3)$-curves one has to 
make into a $(-2)$ to get a triple point.
\end{remark}

\section{Sandwiched singularities}
Sandwiched singularities are normal surface singularities, which admit a 
birational map to $(\C^2,0)$. Following De Jong and Van Straten \cite{tds} 
we describe  them and their deformations in terms of \textit{decorated curves}.

\begin{defn}
Let $C=\bigcup_{i\in B} C_i$ be a plane curve singularity. The number
$m(i)$ is the sum of the multiplicities of the branch $C_i$ in the multiplicity 
sequence of the minimal embedded resolution of $C$.
\end {defn}
\begin{defn}
A \textit{decorated curve} is a curve singularity together with a function
$l\colon B\to \N$ on the set of branches, with the property that
$l(i)\geq m(i)$. A decorated curve $(C,l)$ is non-singular  if
$C$ consists of one smooth branch, and $l(1)=0$.
\end{defn}

\begin{defn}
Let $(C,l)$  be a singular decorated curve on a smooth surface $(Z,p)$ and
let $m_i$ be the multiplicity of the $i$-th branch  $C_i$. Consider the blow 
up $\text{Bl}_p Z\to Z$  of  the singular point.  The \textit{strict transform} 
$(\wl C, \bar l)$ of $(C,l)$ is the decorated curve, consisting 
of the strict transform $\wl C$ of $C$ with decoration $\bar l(i)=l(i)-m_i$.
\end{defn}

Note that we do not allow blow-ups in non-singular decorated curves, for then 
$\bar l$ would become a negative function. Therefore the 
\textit{embedded resolution} of a decorated curve is the unique (minimal) 
composition of point blow-ups, such that the strict transform of the 
decorated curve is non-singular. It is obtained from the minimal 
embedded resolution of $C$ by $l(i)-m(i)$ consecutive point blow-ups 
in each branch  $C_i$. 

\begin{defn}
Let $(C,l)\subset (Z,p)$  be a decorated curve with embedded resolution
$(\wl  C,0)\subset \wt Z(C,l)$. The analytic space $X(C,l)$ is obtained from
$\wt Z(C,l)- \wl C$ by blowing down the maximal compact subset, that is, all
exceptional divisors, not intersecting the strict transfrom $\wl C$.  
\end{defn}

The space $X(C,l)$ can be smooth, or it may have several singularities. Each 
singularity is a sandwiched singularity. Given a sandwiched singularity, it is 
always possible to find a decorated curve $(C,l)$ such that the sandwiched 
singularity is the only singularity of the space  $X(C,l)$. Even then the
representation of a sandwiched singularity as $X(C,l)$ is not unique.

\begin{prop}\label{notsandwich}
A singularity is not sandwiched, if its graph contains $D_4$ as subgraph
or has the following subgraph:
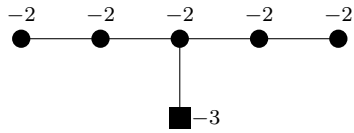
\[
\bp(4,1,3)(0,-1)
\put(0,0){\ci {-2}}  \put(0,0){\lijn}
\put(1,0){\ci {-2}} \put(1,0){\lijn}
\put(2,0){\ci  {-2}}  \put(2,0){\lijn}
\put(3,0){\ci {-2}}  \put(3,0){\lijn}
\put(4,0){\ci {-2}}  
\put(2,0){\line(0,-1){1}} \put(2,-1){\vit{-3}}
\ep
\]
\end{prop}
\begin{proof}
A sandwiched singularity has at least one exceptional curve $E_i$
with $E_i\cdot Z<0$ and multiplicity one in the fundamental cycle $Z$. Indeed, 
the strict transform of the first blown-up curve in the construction has this 
property: the compact part of the divisor of a general linear function is an 
upper bound for $Z$. This criterion excludes $D_4$ and the graph
shown.
\end{proof}

It follows that the singularities of type III.5 and higher in Table \ref{tableA} 
are not sandwiched. Most of the other ones are not excluded by the above 
criterion, and in fact they are sandwiched.

The numbers $l(i)$ determine a divisor on the normalisation $\wl C$ of
$C$. This interpretation allows a more global point of view, in which a
decorated curve $(C,l)$ is a curve $C$ on a smooth surface $Z$ together with a 
divisor $l$ on the normalisation $\wl C$. We form again the embedded resolution
$\wt Z(C,l)$ and blow down the
maximal compact subset of  $\wt Z(C,l)- \wl C$ to obtain the space $X(C,l)$.

We now consider deformations.
Recall that a deformation of a plane curve singularity admits a 
simultaneous normalisation if and only if it is $\delta$-constant.

\begin{defn}
A 1-\textit{parameter deformation} $(C_S,l_S)$ of a decorated curve $(C,l)$ over a 
germ $S$ of a smooth curve is  a deformation $C_S\to S$ of $C$ with 
simultaneous normalisation $\wl C_S\to S$, together with a deformation 
of the divisor $l$, such that each fibre $(C_s,l_s)$ is a decorated curve. 
\end{defn}

\begin{theorem}[{\cite[Thm. 4.4]{tds}}]
The 1-parameter deformations of a sandwiched singularity $X(C,l)$ are exactly
deformations $X(C_S,l_S)$ for 1-parameter deformations $(C_S,l_S)$ of the 
decorated curve $(C,l)$.
\end{theorem}

A simplified proof can be found in the thesis of Konrad M\"ohring \cite{moe}, 
who constructs the deformation  $X(C_S,l_S)$  directly from
$(C_S,l_S)$ by blowing up a family of complete ideals.

To study deformations of taut sandwiched singularities we use a specific 
representation with a decorated curve. We start with
singularities with reduced fundamental cycle, see also \cite[Ex. 1.5 (4)]{tds}.
We construct the resolution graph of a decorated curve. This is an embedded 
resolution graph for the curve $C$, with as usual arrows for the strict 
transforms $\wl C_i$, and decorations $\bar l(i)=0$, which we omit. We choose 
one end of the resolution graph of the singularity, whose   vertex $E_0$ will 
correspond the strict transform of  the first curve blown up. To a  vertex $E_i$  
we connect $-Z\cdot E_i$ $(-1)$-vertices with an arrow attached to it, 
except for $E_0$, to which we connect one vertex and arrow less: 
only  $-Z\cdot E_i-1$ ones.

\begin{example} The cyclic quotient singularity  $X_{37,11}$.
We start at the left end. The resulting graph is
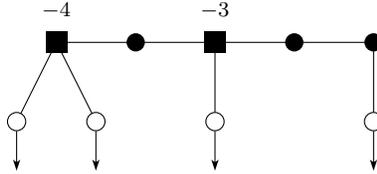
\[
\bp(4,2)(0,-1.7)
\put(0,0){\vi{-4}}  \put(0,0){\lijn}
\put(0,0){\line(1,-2){0.45}} \put(0.5,-1){\circle{0.23}}
\put(0.5,-1.12){\vector(0,-1){0.5}} 
\put(0,0){\line(-1,-2){0.45}} \put(-0.5,-1){\circle{0.23}}
\put(-0.5,-1.12){\vector(0,-1){0.5}} 
\put(1,0){\cir} \put(1,0){\lijn}
\put(2,0){\vi  {-3}}  \put(2,0){\lijn}
\put(3,0){\cir}  \put(3,0){\lijn}
\put(4,0){\cir}  
\put(2,0){\line(0,-1){0.88}} \put(2,-1){\circle{0.23}}
\put(2,-1.12){\vector(0,-1){0.5}} 
\put(4,0){\line(0,-1){0.88}} \put(4,-1){\circle{0.23}}
\put(4,-1.12){\vector(0,-1){0.5}} 
\ep
\]
It blows down to the following decorated curve.
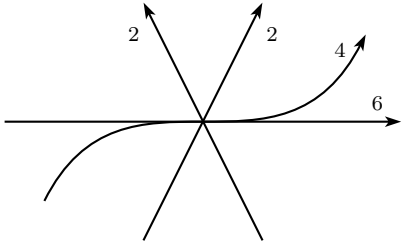
\[
\begin{picture}(6,3)(-3,-1.5)
\thicklines
\put(-2.5,0){\vector(1,0){5}}
\put(2.2,0.15){\makebox(0,0)[b]{$\scriptstyle 6$}}
\put(-.75,-1.5){\vector(1,2){1.5}}\put(.75,-1.5){\vector(-1,2){1.5}}
\cbezier(-2,-1)(-1,1)(1,-1)(2,1)
\put(1.8,0.9){\makebox(0,0)[r]{$\scriptstyle 4$}}
\put(2.05,1.1){\vector(1,2){0}}
\put(0.8,1.1){\makebox(0,0)[l]{$\scriptstyle 2$}}
\put(-0.8,1.1){\makebox(0,0)[r]{$\scriptstyle 2$}}
\end{picture}
\]
\end{example}

For cyclic quotient singularities this representation with smooth branches 
has the property that
$\min\{l(i),l(j)\}=C_i\cdot C_j+1$ for each pair of branches. It was observed 
by M\"ohring \cite{moe} that this property can be used to give a new proof 
of  Riemenschneider's conjecture that cyclic 
quotients deform only into cyclic quotients. 
Koll\'ar and Shepherd-Barron \cite{k-sb} derive  it from the stronger result that in a 
deformation of a rational singularity with reduced fundamental cycle 
the number of ends of the graph cannot go up. 
Their result can also be obtained with the present methods. We first treat the 
cyclic quotient case, where the argument is more transparent. 

\begin{lemma}
Let $(C,l)$ be the germ of a decorated curve with smooth branches with the property
that for each pair of branches
\[
 C_i\cdot C_j\leq \min\{l(i),l(j)\}\leq C_i\cdot C_j+1\;.
\] 
Then the only singularities of the space $X(C,l)$ are cyclic quotients.
\end{lemma}
\begin{proof}
Let $C_r$ be a branch such that $l(i)\leq l(r)$ for all branches $C_i$.
We construct the embedded resolution of $(C,l)$ in two steps. We first 
consecutively blow up $l(r)$ times in the origin of the strict transform 
of the branch $C_r$.  This introduces a chain $E_1,\dots,E_{l(r)}$ of 
exceptional curves.
If $l(i)=C_i\cdot C_r + 1$, we  blow up once in the intersection point of 
$E_{l(i)}$ and the strict transform of $C_i$, and we do not blow up further 
in $C_i$. The newly introduced $(-1)$-curve intersects the strict transform 
of $C_i$ on the minimal resolution and is therefore not part of the exceptional 
set for $X(C,l)$. If  $l(i)=C_i\cdot C_r $, we do not blow up in $C_i$ and 
the curve $E_{l(i)}$ does not belong to the exceptional set. This set is thus 
a subset of the chain $E_1,\dots,E_{l(r)-1}$, which may consist of several 
connected components.
\end{proof}

\begin{prop}
Cyclic quotient singularities deform only into cyclic quotients.
\end{prop}
\begin{proof}
Choose, as above,  a representation $X(C,l)$ with $(C,l)$ a decorated curve 
with smooth branches and the property that
$\min\{l(i),l(j)\}=C_i\cdot C_j+1$ for each pair of branches. 
Let $X(\wt C,\tilde l)$ be a general fibre of a 1-parameter deformation.
Consider a pair of branches. Suppose that $C_i\cdot C_j=n$ and that $l(i)=n+1$. 
Then $\wt C_i\cdot \wt C_j=\sum_p n_p=n$,
where the sum runs over the intersection points. The support of $\tilde l(i)$ 
on $\wt C_i$ may contain other points.
Now
\[
\sum_{p\in \wt C_i \cap \wt C_j} \tilde l_p(i) + \sum_{q\notin \wt C_j}
l_q(i)=l=n+1=1+\sum_{p\in \wt C_i \cap \wt C_j} n_p \;.
\]
Because always $\tilde l_p(i)\geq n_p$ we see that for at most one point
$\tilde l_p(i)= n_p+1$ while for the others $\tilde l_p(i)= n_p$.
So for each singularity $p$ of $(\wt C,\tilde l)$ the property
$( \wt C_i\cdot \wt C_j)_p\leq \min\{\tilde l_p(i),\tilde l_p(j)\}\leq 
(\wt C_i\cdot \wt C_j)_p+1$
holds.
\end{proof}

\begin{lemma}
Let $(C,l)$ be the germ of a decorated curve with smooth branches.
Suppose that the set of branches $B$ can be written as the {\rm(}not necessarily 
disjoint\/{\rm)} union $B_1\cup\dots\cup B_k$ such that for all $1\leq m\leq k$ the 
property
\[
 C_i\cdot C_j\leq \min\{l(i),l(j)\}\leq C_i\cdot C_j+1\;
\] 
holds for all pairs $(i,j)\in B_m\times B_m$.
Then the number of ends of the singularities of the space $X(C,l)$ 
is at most $k+1$.
\end{lemma}

\begin{proof}
Again we construct the embedded resolution of $(C,l)$ in two steps.
For each subset $B_m$ we choose a branch $C_m$ with $l(m)$ maximal.
The first step is to  resolve the curve $\cup_m C_m$. As this curve has $k$ 
branches, the resulting embedded resolution graph has (at most) $k+1$ ends. 
The exceptional curves of the additional blow-ups needed to resolve $(C,l)$ 
are not exceptional for $X(C,l)$.
\end{proof}

\begin{prop}
In a deformation of a rational singularity with reduced fundamental cycle 
the number of ends cannot increase.
\end{prop}
\begin{proof}
We choose a representation with a decorated curve with smooth branches.
For each end of the graph of the singularity (except the root) we choose a 
curve $C_m$, whose strict transform is connected to this end by a
$(-1)$-curve. The set $B_m$ contains all branches, which are connected 
by $(-1)$-curves to the chain from the root to $C_m$. Then 
$\min\{l(i),l(j)\}= C_i\cdot C_j+1$ for all $i,j\in B_m$. As before, we deduce 
that $( \wt C_i\cdot \wt C_j)_p\leq \min\{\tilde l_p(i),\tilde l_p(j)\}\leq 
(\wt C_i\cdot \wt C_j)_p+1$ for each singular point of the deformed 
curve $(\wt C,\tilde l)$, through which branches in the set $B_m$ pass. 
Therefore we have a division in at most $k$ sets, and each singularity of  
$X(\wt C,\tilde l)$ has at most $k+1$ ends.
\end{proof}

\begin{remark}
For non-reduced fundamental cycle the number of ends can increase.
Examples are provided by the deformations of Lemma \ref{lemetilde}.
We give a sandwich description of the deformation for the case of
a surface singularity of type $\wt E_6$ with $-b_{1,1}=-4$ and all other curves $-2$.
It has a sandwiched representation with decorated curve $(E_{12}, 12)$, 
where $E_{12 }$ is the curve
$x^3+y^7+axy^5=0$. It deforms into $(\wt E_7,(4,4,4))$, giving a $2$-star.
\end {remark}

Next we study sandwiched singularities in the classes  III.2 and III.3. 
The singularities of type III.2 whose graph contains a $D_4$ subgraph, are 
not sandwiched, but they are simple as they are dihedral quotients. The 
sandwiched ones can be seen as special case of the type III.3, if we allow 
the arms to be shorter.
This means that we are looking at graphs of the form:
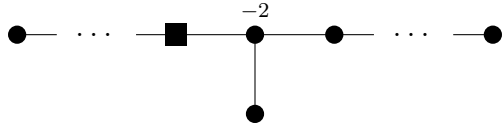
\[
\bp(6,1,5)(1,-1)
\put(1,0){\cir}  \put(1,0){\keten}
\put(3,0){\vir}  \put(3,0){\lijn}
\put(4,0){\ci{-2}}  
\put(4,0){\lijn}\put(5,0){\cir}
\put(5,0){\keten}\put(7,0){\cir}
\put(4,0){\line(0,-1){1}} 
\put(4,-1){\cir}
\ep
\]

\begin{prop}\label{propdefonly}
Sandwiched singularities with graph as above deform only to singularities of the 
same type or to singularities with reduced fundamental cycle and at most three ends.
\end{prop}
\begin{proof}
We start by describing  a decorated curve $(C,l)$.
 We need some notation.
The  left arm of the graph is $E_{1,1}+\dots+E_{1,k}$, with $E_{1,k}$ the $(-3)$.
The right arm is $E_{2,1}+\dots+E_{2,s+1}$ and the short arm consists
only of $E_{3,1}$. 
First look at the triple point graph of this form, having exactly one $(-3)$
and further only $(-2)$'s. A decorated curve giving 
this graph is  $(A_{2k}, 2k+4+s)$. We call the curve $C_0$.
To make the self-intersections more negative we add branches. They come 
in three types, one for each arm. 
On the left arm, if a $(-1)$ intersects $E_{1,m}$, then we have a 
smooth branch $C_i$ with $C_0\cdot C_i= 2m$ and $l(C_i)=m+1$.
We make the short arm $E_{3,1}$ more negative with a smooth
branch $C_i$ with $C_0\cdot C_i= 2k+1$ and $l(C_i)=k+2$.
Finally, if  a $(-1)$ intersects $E_{2,n}$ on the right arm, then we have
a branch $C_i$ of type $A_{2k}$ with  $C_0\cdot C_i= 4k+2+n$ and $l(C_i)=2k+3+n$.

We use induction on the number of branches of $(C,l)$.
If the curve consists only of $(C_0,l(0))$ the claim is true
as an $A_{2k}$ deforms with $\delta$-const into one $A_{2l}$, $0\leq l<k$
and some $A_{2m_i-1}$ with $k=l+\sum m_i$. The $(A_{2l}, 2l+1+t)$
gives at most a smaller graph of the same type, while 
$X(A_{2m-1},(m+t_1,m+t_2))$ has reduced fundamental cycle and at most three ends.

Now consider a deformation $X(\wt C, \tilde l)$  of $X(C,  l)$ where
$C$ has several branches. Let $C_i$ be a branch different from $C_0$.
We wish to compare the singularities of $X(\wt C, \tilde l)$ and
$X(\wt C\setminus \wt C_i, \tilde l)$.
The Proposition follows from the following claim.

\begin{claim}
The  graph of  the singularities of $X(\wt C,\tilde l)$
is a subgraph of the graph of $X(\wt C\setminus \wt C_i,\tilde l)$
or it is obtainable from a subgraph by making some self-intersections more negative. 
\end{claim}

We first resolve the decorated curve  $(\wt C\setminus \wt C_i,\tilde l)$.
If this resolution also resolves  $(\wt C,\tilde l)$, i.e., no further blow ups on the 
strict transform of $\wt C_i$ are needed, then the exceptional curves intersected 
by the strict transform of $\wt C_i$ are not exceptional for 
$X(\wt C,\tilde l)$, and the graph of  the singularities of $X(\wt C,\tilde l)$
is a subgraph of the graph of $X(\wt C\setminus \wt C_i,\tilde l)$. 

We shall show that if this is not the case, then there is exactly one extra 
blow-up needed. If the center of the blow up does not lie on an exceptional 
divisor, we get just one $(-1)$-curve and no new singularity.
If the center is a smooth point on an irreducible component $E_a$ of the 
exceptional divisor, then the self intersection $E_a\cdot E_a$
is made more negative. The graph of the singularity in question is therefore 
obtainable from a subgraph of the  graph of $X(\wt C\setminus \wt C_i,\tilde l)$.
Finally, if the center is an intersection point of two divisors $E_a$ and $E_b$, then 
the newly introduced $(-1)$, being
not exceptional for 
$X(\wt C,\tilde l)$, breaks up the graph; the self-intersections of 
$E_a$ and $E_b$ become more negative.
Therefore the graphs for $X(\wt C,\tilde l)$ are obtainable from
subgraphs of the graph of $X(\wt C\setminus \wt C_i,\tilde l)$.

We look only at the intersection of a branch $\wt C_i$ with $\wt C_0$.
If $C_i$ is  a smooth branch with $C_0\cdot C_i= 2m$ and $l(C_i)=m+1$,
then one has to blow up in $m+1$ points to resolve $(\wt C_i,m+1)$.
Let $m_{0,p}$
be the multiplicity of  $\wt C_0$ in such a point $p$.
As $\wt C_0\cdot \wt C_i = \sum _p m_{0,p}=2m$,  it follows that either 
$m_{0,p}=2$ for $m$ points and one point does not lie on $\wt C_0$, or all 
$m+1$ points lie on $\wt C_0$ and for two of them $m_{0,p}=1$. 
A similar argument settles the case of  a smooth
branch $C_i$ with $C_0\cdot C_i= 2k+1$ and $l(C_i)=k+2$.

Finally we look at
a branch $C_i$ of type $A_{2k}$ with  $C_0\cdot C_i= 4k+2+n$ and $l(C_i)=2k+3+n$. 
To resolve $(\wt C_i,l(\wt C_i))$ we need $k+3+n$  blow ups, and $\wt C_i$ 
has multiplicity $2$ in $k$ of them. In these points the multiplicity
of $\wt C_0$ can be 2, 1 or 0. 
As intersection multiplicity $\wt C_0\cdot \wt C_i$ is 
$4k+2+n$, there are only two possibilities. The first one is 
$k\times 4 +  (n+1)\times 1 + 0$, the same numbers as for 
$C_0\cdot  C_i$, or in one multiplicity 2 point of $\wt C_i$ the multiplicity of 
$\wt C_0$ is 1. Then $4k+2+n=(k-1)\times 4 +2+2+  (n+2)\times 1 $.

Therefore in all cases at most one point blown up does not lie on the 
strict transform of $\wt C_0$.
\end{proof}

Singularities with a graph of type III.4 are not sandwiched if the graph has
the graph of Proposition \ref{notsandwich} as subgraph. To be sandwiched  
some vertex weights on this subgraph have to be more negative. We distinguish 
three types, with different sandwich representation.
Let the left arm be $E_{1,2}+E_{1,1}$, where always $-b_{1,1}=-2$,
the short arm $E_{3,1}$ and the right arm  $E_{2,1}+E_{2,2}+\dots+E_{2,k}$. 
For each of the three types the singularity of lowest multiplicity is a quadruple 
point, realisable as $X(C_0,l)$ with $C_0$ irreducible. 

For the first type  $-b_{3,1}=-4$ (for the quadruple point).
We take as decorated curve $(E_6,k+7)$. To make the self-intersection of 
$E_{3,1}$ (the first blown 
up curve) more negative  we add one or more smooth branches 
with $l=2$, for  $E_{1,2}$ smooth branches with $l=3$, and finally for  
$E_{2,t}$ on the right arm one or more branches $(E_6,t+7)$, intersecting 
$C_0$ with  multiplicity $12+t$.

The second type has always $-b_{3,1}=-3$. If also $-b_{1,2}=-3$, we get
the quadruple point with the irreducible curve $(E_8,k+8)$.
To make $E_{1,2}\cdot E_{1,2}$  more negative we add one or more smooth branches 
with $l=2$ and for  $E_{2,t}$ on the right arm one or more branches 
$(E_8,t+8)$, intersecting  $C_0$ with 
multiplicity $15+t$.

For the last type always $-b_{3,1}=-3$ and $-b_{1,2}=-2$.
In this case  $E_{2,k}$ is the first blown up curve. For the quadruple point
the only other $(-3)$-curve besides $E_{3,1}$ is $E_{2,2}$. We take a curve
equisingular with $(x^3+y^{3k-1},4+3k)$; here we may assume that $k>2$, as 
$k=2$ is dealt with in the previous case. The only curves, which can be made 
more negative, are on the right arm, and we use smooth branches for this purpose.

\begin{prop}
Sandwiched singularities with graph of type\/ {\rm III.4} deform only to singularities of  
type  {\rm I, II} or\/ {\rm III.1} up till\/  {\rm III.4}.
\end{prop}
\begin{proof}
We use again induction on the number of branches. For the induction start we 
have to describe $\delta$-const deformations of $(C_0, l)$. For $E_6$ and $E_8$ 
this is not difficult, but for the third type it is not so easy. Therefore we use a 
different argument.
By Proposition \ref{proprqp}  quadruple points, obtainable from triple points, 
are simple. As sandwiched singularities deform only into sandwiched singularities, 
we get the statement of the Proposition.

For the induction step we use the same claim as in the proof of Proposition \ref{propdefonly}.
As before the result follows if we can prove that either the resolution 
$(\wt C\setminus \wt C_i,\tilde l)$
is  also the resolution of  $(\wt C,\tilde l)$,  or that only
one extra blow up in a smooth point of the strict transform of $\wt C_i$ is needed.
For smooth branches this is the same argument as before. It remains to look at 
branches of type $E_6$ and $E_8$. Both cases being similar, we do here 
only the last one.

If $E_8$ deforms $\delta$-const to a collection of $A_k$ singularities, then the
multiplicities of the (infinitely near) points in the minimal embedded  resolution 
are $(2,2,2,2)$, whereas they are $(3,2)$ if there is a triple point.
This can be checked from the list of possible combinations, but it is all the 
information we need here.
Therefore the multiplicities in the resolution of 
 $(E_8,t+8)$ are $(3,2, 1^{t+3})$ or $(2^4,1^t)$.
The intersection multiplicity with $\wt C_0$ has to be $15+t$.
If the multiplicities in the minimal resolution of $\wt C_0$ are $(3,2)$
then we can get $15+t$ as $9+4+(t+2)+0$, or $9+2+2+(t+2)$, but not
from a $\wt C_i$ with multiplicities $(2^4,1^t)$, as
$6+4+2+2+t=14+t$. If  $\wt C_0$ is of type $(2^4)$, then we can get
$6+4+2+2+(t+1)$, or $4\times4+(t-1)+0$ or $3\times 4+2+2+(t-1)$. 
So there is at most one smooth  point of the strict transform of
$\wt C_i$ on the resolution of $(\wt C_0,\tilde l)$, which still has to be blown up. 
\end{proof}

\begin{cor}
Sandwiched singularities are simple if and only if they are taut and 
quasi-homogeneous.
\end{cor}


\begin{thebibliography}{99}



\bibitem{ar}
V. I. Arnol'd, 
\textit{Normal forms of functions near degenerate critical points, the Weyl
 groups $A_k$, $D_k$, $E_k$ and Lagrangian singularities}.
(Russian)  Funkcional. Anal. i Prilo\v zen.  \textbf{26}  (1972),  no. 4, 3--25.

\bibitem{BrGa}
J. W. Bruce and T. J. Gaffney, 
\textit{Simple singularities of mappings $\C,0\to\C^2,0$}.
 J. London Math. Soc. (2) \textbf{26}  (1982),   465--474.

\bibitem{EV}
H\'el\`ene Esnault and Eckart Viehweg, 
\textit{Two-dimensional quotient singularities deform to quotient 
singularities}. Math. Ann. \textbf{271} (1985), 439--449. 


\bibitem{FN}
Anne Fr\"uhbis-Kr{\"u}ger and Alexander Neumer, 
\textit{Simple Cohen-Macaulay codimension 2 singularities}.
Comm. Algebra \textbf{38} (2010), 454--495.

\bibitem{gab}
A. M. Gabrielov, 
\textit{Bifurcations, Dynkin diagrams and the modality of isolated
 singularities}.
(Russian)  Funkcional. Anal. i Prilo\v zen.  \textbf{8}   (1974),  no. 2, 7--12.

\bibitem{gi}
Marc Giusti, 
\textit{Classification des singularités isolées simples d'intersections
 complètes}.
In:  Singularities, Part 1 (Arcata, Calif., 1981), 
pp. 457--494, Proc. Sympos. Pure Math. \textbf{40},  Amer. Math. Soc., Providence, 
RI,  1983. 

\bibitem{GD}
Gert-Martin Greuel, Nguyen Hong Duc,
\textit{Right simple singularities in positive characteristic}.  \texttt{arXiv:1206.3742} 

\bibitem{tdq}
Theo de Jong and Duco van Straten, 
\textit{On the base space of a semi-universal deformation of rational quadruple points}.
Ann. of Math. (2) \textbf{134} (1991), 653–678. 

\bibitem{tds}
T. de Jong and D. van Straten,  
\textit{Deformation theory of sandwiched singularities}.
 Duke Math. J.   \textbf{95}  (1998),   451--522.

\bibitem{ka}
Ulrich Karras, 
\textit{Normally flat deformations of rational and minimally
elliptic singularities}.  In:
Singularities, Part 1 (Arcata, Calif., 1981),  pp. 619--639,
Proc. Sympos. Pure Math. \textbf{40}, Amer. Math. Soc., Providence, RI, 1983.

\bibitem{k-sb}
J. Koll\'ar and N. I. Shepherd-Barron,
\textit{Threefolds and deformations of surface singularities}.
Invent. Math.  \textbf{91} (1988), 299--338.


\bibitem{la1}
  Henry B. Laufer,
  \textit{Taut Two-Dimensional Singularities}.
  Math. Ann.  \textbf{205} (1973), 131--164.



\bibitem{la3}
  Henry B. Laufer,
  \textit{Ambient Deformations for Exceptional Sets in Two-Manifolds}.
  Invent. math.  \textbf{55} (1979), 1--36.

\bibitem{le}
L\^e D\~ung Tr\`ang,
\textit{Les singularit\'es sandwich}.  
In: Resolution of singularities (Obergurgl, 1997), pp. 457–483,
Progr. Math. \textbf{181}, Birkh\"auser, Basel, 2000.

\bibitem{moe}
Konrad Möhring, 
\textit{On Sandwiched Singularities}.
Dissertation, Mainz 2004.\\
URN: \texttt{urn:nbn:de:hebis:77-4878}

\bibitem{pal}
V. P. Palamodov,   \textit{Moduli in versal deformations of complex spaces}.
(Russian)  Dokl. Akad. Nauk SSSR  \textbf{230}  (1976),  no. 1, 34--37.

\bibitem{st-p}
Jan Stevens,  
\textit{Partial resolutions of rational quadruple points}. 
Internat. J. Math. \textbf{2} (1991), 205--221.



\bibitem{vi}
E. B. Vinberg,  
\textit{Complexity of actions of reductive groups}.
(Russian)  Funktsional. Anal. i Prilozhen.  \textbf{20}  (1986),  
no. 1, 1--13.


\bibitem{wa}
Jonathan M. Wahl,   
\textit{Simultaneous resolution and discriminantal loci}.
 Duke Math. J.  \textbf{46}  (1979),   341--375.




\bibitem{wall}
C. T. C. Wall, 
\textit{Classification of unimodal isolated singularities of complete
 intersections}.
 Singularities, Part 2 (Arcata, Calif., 1981), 
pp.  625--640, Proc. Sympos. Pure Math. \textbf{40}, 
Amer. Math. Soc., Providence, RI,  1983. 


\end{thebibliography}
\end{document}